\documentclass[11pt,a4paper]{article}

\usepackage{amsmath,amsthm,amssymb,color}
\usepackage[numbers]{natbib}

\usepackage{bbm}
\usepackage{mathrsfs}
\usepackage[breaklinks=true]{hyperref}

\numberwithin{equation}{section}
\allowdisplaybreaks[4]

\theoremstyle{plain}
\newtheorem{theorem}{Theorem}[section]
\newtheorem{proposition}[theorem]{Proposition}
\newtheorem{lemma}[theorem]{Lemma}

\theoremstyle{definition}

\newtheorem{remark}[theorem]{Remark}

\makeatletter

\newcount\minute
\newcount\hour
\newcount\hourMins
\def\now{%
\minute=\time%
\hour=\time \divide \hour by 60%
\hourMins=\hour \multiply\hourMins by 60%
\advance\minute by -\hourMins%
\zeroPadTwo{\the\hour}:\zeroPadTwo{\the\minute}%
}
\def\zeroPadTwo#1{\ifnum #1<10 0\fi#1}

\renewcommand{\cite}{\citet*}

\def\^#1{\ifmmode {\mathaccent"705E #1} \else {\accent94 #1} \fi}
\def\~#1{\ifmmode {\mathaccent"707E #1} \else {\accent"7E #1} \fi}

\def\*#1{#1^\ast}
\edef\-#1{\noexpand\ifmmode {\noexpand\bar{#1}} \noexpand\else \-#1\noexpand\fi}
\def\>#1{\vec{#1}}
\def\.#1{\dot{#1}}

\def\atop{\@@atop}
\def\*#1{\mathscr{#1}}

\renewcommand{\leq}{\leqslant}
\renewcommand{\geq}{\geqslant}
\renewcommand{\phi}{\varphi}
\newcommand{\eps}{\varepsilon}
\renewcommand{\eps}{\varepsilon}
\newcommand{\D}{\Delta}
\newcommand{\eq}{\eqref}

\def\dist{\mathop{\mathrm{dist}}}

\newcommand{\IE}{\mathbbm{E}}

\newcommand{\IP}{\mathbbm{P}}
\newcommand{\Var}{\mathop{\mathrm{Var}}\nolimits}
\newcommand{\Cov}{\mathop{\mathrm{Cov}}}

\newcommand{\law}{\mathscr{L}}

\newcommand{\IR}{\mathbbm{R}}

\newcommand{\II}{\mathbbm{I}}
\newcommand{\Id}{\mathrm{Id}}

\def\be#1{\begin{equation*}#1\end{equation*}}
\def\ben#1{\begin{equation}#1\end{equation}}
\def\bes#1{\begin{equation*}\begin{split}#1\end{split}\end{equation*}}
\def\besn#1{\begin{equation}\begin{split}#1\end{split}\end{equation}}

\def\ban#1{\begin{align}#1\end{align}}

\def\bklr#1{\bigl(#1\bigr)}

\def\bbbklr#1{\biggl(#1\biggr)}

\def\abs#1{\vert#1\vert}
\def\babs#1{\bigl\vert#1\bigr\vert}

\def\Id{\mathbbm{I}}

\def\beqn#1\eeqn{\begin{align}#1\end{align}}
\def\beq#1\eeq{\begin{align*}#1\end{align*}}

\usepackage{graphicx}
\usepackage{latexsym}
\usepackage{amsmath,amsthm,amssymb,amscd}
\usepackage{epsf,amsmath}

\def\E{{\IE}}

\def\P{{\IP}}

\def\min{{\rm min}}
\def\I{{\rm I}}

\newcommand{\mr}{\IR^d}

\renewcommand\section{\@startsection {section}{1}{\z@}%
{-3.5ex \@plus -1ex \@minus -.2ex}%
{1.3ex \@plus.2ex}%
{\center\small\sc\mathversion{bold}\MakeUppercase}}

\def\subsection#1{\@startsection {subsection}{2}{0pt}%
{-3.5ex \@plus -1ex \@minus -.2ex}%
{1ex \@plus.2ex}%
{\bf\mathversion{bold}}{#1}}

\def\subsubsection#1{\@startsection{subsubsection}{3}{0pt}%
{\medskipamount}%
{-10pt}%
{\normalsize\itshape}{\kern-2.2ex. #1.}}

\def\blfootnote{\xdef\@thefnmark{}\@footnotetext}

\makeatother

\begin{document}

\title{\sc\bf\large\MakeUppercase{A multivariate CLT for bounded decomposable random vectors with the best known rate}}
\author{\sc Xiao Fang \footnote{Department of Statistics and Applied Probability, 
Block S16, Level 7, 6 Science Drive 2,
Faculty of Science, 
National University of Singapore,
Singapore 117546}}
\date{\it National University of Singapore}
\maketitle

\begin{abstract} 

We prove a multivariate central limit theorem with explicit error bound in a non-smooth function distance for sums of  bounded decomposable $d$-dimensional random vectors. The decomposition structure is similar to that of Barbour, Karo\'nski and Ruci\'nski (1989) and is more general than the local dependence structure considered in Chen and Shao (2004). The error bound is of the order $d^{\frac{1}{4}} n^{-\frac{1}{2}}$, where $d$ is the dimension and $n$ is the number of summands. The dependence on $d$, namely $d^{\frac{1}{4}}$, is the best known dependence even for sums of independent and identically distributed random vectors, and the dependence on $n$, namely $n^{-\frac{1}{2}}$, is optimal. We apply our main result to a random graph example.
\end{abstract}

\vspace{9pt}
\par

\vspace{9pt}
\par

\section{Introduction}

Let $W=\frac{1}{\sqrt{n}}\sum_{i=1}^n X_i$ be a standardized sum of independent and identically distributed $d$-dimensional random vectors such that $\E (X_1)=0$, $\Cov(X_1)=\II_d$ and $\E( |X_1|^3)<\infty$ where $\II_d$ denotes the $d$-dimensional identity matrix and $\abs{\cdot}$ denotes the Euclidean norm of a vector.
\cite{Be03} proved the following bound on a non-smooth function distance between the distribution of $W$ and the standard $d$-dimensional Gaussian distribution:
\ben{\label{206}
\sup_{A\in \mathcal{A}}|\P (W\in A)-\P (Z \in A)| \leq C \E( |X_1|^3) d^{\frac{1}{4}} n^{-\frac{1}{2}}
}
where $\mathcal{A}$ denotes the collection of all the convex sets in $\IR^d$,
$Z$ is a $d$-dimensional standard Gaussian vector, and $C$ is an absolute constant.

In this paper, we aim to prove a multivariate central limit theorem with an error bound having the same order of magnitude in terms of $d$ and $n$ for the same non-smooth function distance as in \eq{206}, but for general bounded dependent random vectors.
The dependence structure we consider is similar to that of \cite{BaKaRu89} and is more general than the local dependence structure considered in \cite{ChSh04}. The approach we use is the recursive approach in Stein's method for multivariate normal approximation.

Stein's method was introduced by \cite{St72} for normal approximation and has become a powerful tool in proving distributional approximations. We refer to \cite{BaCh05} for an introduction to Stein's method. Stein's method for multivariate normal approximation was first studied in \cite{Goetze1991}. Along with an inductive argument for sums of independent random vectors, he proved an error bound whose dependence on the dimension is $d^{\frac{3}{2}}$ for the distance in \eq{206}. The dependence on the dimension was improved to $d^{\frac{1}{4}}$ in \cite{Be03} (cf. \eq{206}) and \cite{Be05} by using the same inductive approach and a Lindeberg-type argument. Using the recursive approach in Stein's method, \cite{RiRo96} proved a multivariate normal approximaton result for sums of bounded random vectors that allow for a certain decomposition. The error bound they obtained for the distance in \eq{206} is typically of the order $O_d( n^{-\frac{1}{2}} \log n)$ with unspecified dependence on $d$ and an additional logarithmic factor. Recently, \cite{FaRo14} proved a multivariate normal approximation result under the general framework of Stein coupling (cf. \cite{ChRo10}). For sums of locally dependent bounded random vectors, their bound for the distance in \eq{206} is typically of the order $d^{\frac{7}{4}} n^{-\frac{1}{2}}$. Compared with the existing literature on bounding the non-smooth function distance in \eq{206} for multivariate normal approximation for sums of bounded random vectors, our new result not only applies to a decomposition structure more general than local dependence, but also obtains an error bound typically of the same order as for sums of independent and identically distributed random vectors in \eq{206}. Stein's method has also been used to prove smooth function distances for multivariate normal approximation under various dependence structures, see, for example, \cite{Goldstein1996}, \cite{Raic2004},
\cite{ChMe08} and \cite{Reinert2009}.

Many problems in random graph counting satisfy the decomposition structure considered in this paper.
See, for example, \cite{BaKaRu89}, \cite{JaNo91}, \cite{AvBe93}, and \cite{RiRo96}.
We will study an example from \cite{RiRo96} and show that the error bound obtained by our main theorem is better than that by \cite{RiRo96}.

The paper is organized as follows. In the next section, we state our main result. In Section 3, we study a random graph counting problem. In Section 4, we prove our main result. Throughout this article, $\abs{\cdot}$ denotes the Euclidean norm of a vector or the cardinality of a set,
and $\Id_d$ denotes the $d$-dimensional identity matrix.
Define $[n]:=\{1,\dots, n\}$ and $X_N:=\sum_{i\in N} X_i$ for an index set $N$.

\section{Main results}

For a sum of standardized $d$-dimensional random vectors $W=\sum_{i=1}^n X_i$ with a certain decomposition structure, we aim to bound the quantity
\ben{\label{1}
  d_{c}\bklr{\law (W), \law (Z)}
   =\sup_{A\in \mathcal{A}}|\P (W\in A)-\P (Z \in A)|,
}
where $Z$ has standard $d$-dimensional Gaussian distribution and $\mathcal{A}$ denotes the collection of all the convex sets in $\IR^d$. The following is our main theorem.

\begin{theorem}\label{t1}
Let $W=\sum_{i=1}^n X_i$ be a sum of $d$-dimensional random vectors such that $\E (X_i)=0$ and $\Cov(W)=\II_d$.
Suppose $W$ can be decomposed as follows:
\ben{\label{t1-1}
\forall\  i\in [n],\  \exists\  i\in N_i\subset [n]\ 
 \text{such that}\  W-X_{N_i}\  \text{is independent of}\  X_i;
}
\ben{\label{t1-2}
\forall\  i\in [n], j\in N_i,\ \exists\  N_i\subset N_{ij}\subset [n]\ 
 \text{such that}\  W-X_{N_{ij}}\  \text{is independent of}\  \{X_i, X_j\};
}
\ben{\label{t1-3}
\forall\  i\in [n], j\in N_i, k\in N_{ij},\ \exists\  N_{ij}\subset N_{ijk} \subset [n]\ 
 \text{such that}\  W-X_{N_{ijk}}\  \text{is independent of}\  \{X_i, X_j, X_k\}.
}
Suppose further that for each $i\in [n], j\in N_i$ and $k\in N_{ij}$,
\ben{\label{t1-4}
|X_i|\leq \beta, \ |N_i|\leq n_1,\  |N_{ij}|\leq n_2,\  |N_{ijk}|\leq n_3.
}
Then there is a universal constant $C$ such that
\ben{\label{t1-5}
d_c(\law(W), \law(Z))\leq C d^{1/4} n \beta^3 n_1(n_2+\frac{n_3}{d}),
}
where $d_c$ is defined as in \eq{1} and $Z$ is a $d$-dimensional standard Gaussian random vector.
\end{theorem}
\begin{remark}
Under the conditions of Theorem \ref{t1} but with $\Cov(W)=\Sigma$, by considering $\Sigma^{-1/2}W$, we have
\be{
d_c(\law(W), \law(\Sigma^{1/2} Z))\leq C d^{1/4} n  ||\Sigma^{-1/2}||^3\beta^3 n_1(n_2+\frac{n_3}{d}),
}
where $||\Sigma^{-1/2}||$ is the operator norm of $\Sigma^{-1/2}$.
\end{remark}
\begin{remark}
The decomposition \eq{t1-1} and \eq{t1-2} is the same as that of \cite{BaKaRu89}, and, as observed there, 
is more general than the local dependence structure studied in \cite{Ch78} and later in \cite{ChSh04}.
Condition \eq{t1-3} is a natural extension of the decomposition of \cite{BaKaRu89}. We need this extra condition to obtain the bound in \eq{t1-5}.
\end{remark}
\begin{remark}
As mentioned at the beginning of the Introduction, the dependence on the dimension $d$ in \eq{t1-5} is the best known dependence even for sums of independent and identically distributed random vectors. In typical applications, $\beta$ is of the order $O(n^{-\frac{1}{2}})$. From the decomposition \eq{t1-1}--\eq{t1-3}, the neighborhood sizes $n_2$ and $n_3$ are typically of the same order as $n_1$. On the other hand, in the local dependence structure studied in \cite{Ch78} and \cite{ChSh04}, quantities corresponding to $n_2$ and $n_3$ are typically of the order $O(n_1^2)$ and $O(n_1^3)$ respectively. \cite{RiRo96} proved a bound similar to \eq{t1-5} for a different decomposition structure. Their bound does not have explicit dependence on $d$ and has an additional $\log n$ term. Recently, \cite{FaRo14} proved a general multivariate central limit theorem for the non-smooth function distance under the framework of Stein coupling (cf. \cite{ChRo10}) with boundedness conditions. Although their bound is more widely applicable, it does not yield optimal dependence on $d$ and does not directly apply to the decomposition structure considered in this paper.
\end{remark}

\section{An application}

Let $n\geq 2, m\geq 1, d\geq 2$ be positive integers. Consider a regular graph with $n$ vertices and vertex degree $m$. Let $N=nm/2$ be the total number of edges. We color each vertex independently with one of the colors $c_i$, $1\leq i\leq d$, with the probability of~$c_i$ being~$\pi_i$, where $\sum_{i=1}^d \pi_i=1$. For $i\in \{1,\dots, d\}$, let $W_i$ be the number of edges connecting vertices both of color $c_i$. Formally, with edges indexed by $j\in \{1,\dots, N\}$, we set
\be{
W_i=\sum_{j=1}^N X_{ji},
}
where $X_{ji}$ is the indicator of the event that the edge $j$ connects two vertices both of color $c_i$.
Let
\ben{\label{301}
W=(W_1,\dots, W_d)^t.
}
Let $\lambda=\E (W)$, $\Sigma=\Cov(W)$. It is known that
\be{
\lambda=(N\pi_1^2,\dots, N\pi_d^2)^t,
}
and (cf. (3.1) of \cite{RiRo96})
\be{
\Var(W_i)=N\pi_i^2(1-\pi_i^2)+2N(m-1)(\pi_i^3-\pi_i^4),
}
\be{
\Cov(W_i, W_j)=-N(2m-1)\pi_i^2 \pi_j^2, \ \text{for}\ i\ne j.
}
We prove the following bound on the non-smooth function distance between the standardized distribution of $W$ and the standard $d$-dimensional Gaussian distribution.
\begin{proposition}\label{pro1}
Let $W$ be defined as in \eq{301}. Let $\lambda$ and $\Sigma$ be the mean and covariance matrix of $W$ respectively. We have
\ben{\label{pro1-1}
d_c (\mathcal{L}(\Sigma^{-1/2}(W-\lambda)), \mathcal{L}(Z) )\leq Cd^{7/4}m^{3/2} L^3 n^{-1/2}
}
where $d_c$ is defined as in \eq{1}, $Z$ is a $d$-dimensional standard Gaussian vector, $C$ is an absolute constant,
and
\be{
L=[\min_{1\leq i\leq d} \{\pi_i^2(1-\pi_i)\}]^{-1/2}.
}
\end{proposition}
\begin{remark}
\cite{RiRo96} proved an upper bound for the left-hand side of \eq{pro1-1} as follows (cf. (3.2) of \cite{RiRo96}):
\ben{\label{302}
d_c (\mathcal{L}(\Sigma^{-1/2}(W-\lambda)), \mathcal{L}(Z) )\leq c_d m^{3/2} L^3 (|\log L|+\log n)n^{-1/2},
}
where $c_d$ is an unspecified constant depending on $d$.
Compared to \eq{302}, our bound in \eq{pro1-1} has explicit dependence on $d$ and does not have the logarithmic terms. 


\end{remark}
\begin{proof}[Proof of Proposition \ref{pro1}]
Recall that
\be{
\Sigma^{-1/2} (W-\lambda)=\sum_{j=1}^N \Sigma^{-1/2} \left( (X_{j1},\dots, X_{jd})^t-(\pi_1^2,\dots, \pi_d^2)^t  \right)=:\sum_{j=1}^N \xi_j.
}
Observe that at most one of $\{X_{j1},\dots, X_{jd}\}$ can be non-zero. Together with $\sum_{i=1}^d \pi_i=1$ and the fact that the elements of $\Sigma^{-1/2}$ are bounded in modulus by $N^{-1/2} L$ (cf. page 339 of \cite{RiRo96}), we have
\be{
|\xi_j|\leq 2 d^{1/2} N^{-1/2} L.
}
Moreover, the summation $\sum_{j=1}^N \xi_j$ can be easily seen to satisfy the decomposition structure \eq{t1-1}--\eq{t1-3} with
\be{
|N_i|\leq 2m, \quad |N_{ij}|\leq 3m, \quad |N_{ijk}|\leq 4m.
}
By applying Theorem \ref{t1}, we obtain the bound \eq{pro1-1}.
\end{proof}

\section{Proof of main theorem}

For given test function $h$, we consider the Stein equation
\ben{   \label{27}
  \D f(w)-w^t\nabla f(w)=h(w)-\E [h(Z)],\qquad w\in \IR^d,
}
where $\D$ denotes the Laplacian operator and $\nabla$ the gradient operator.
If $h$ is not continuous (as for the indicator function of a convex set), then $f$
is not smooth enough to apply Taylor's expansion to the necessary degree, so
more refined techniques are necessary. 

We follow the smoothing technique of \cite{Be03}. Recall that
$\mathcal{A}$ is the collection of all the convex sets in $\IR^d$. For
$A\in\mathcal{A}$, let $h_A(x)=I_A(x)$, and define the
smoothed function
\ben{   \label{28}
  h_{A,\eps}(w) = \psi\bbbklr{\frac{\dist(w, A)}{\eps}},
}
where $\dist(w,A) = \inf_{v\in A}\abs{w-v}$ and
\ben{
  \psi(x)=
  \begin{cases}
    1, & x<0,\\
    1-2x^2, & 0\leq x<\frac{1}{2},\\
    2(1-x)^2, & \frac{1}{2} \leq x <1,\\
    0, & 1\leq x.
\end{cases}
}
Define also 
\be{
  A^\eps = \{x\in\IR^d\,:\, \dist(x,A)\leq \eps\},\qquad
  A^{-\eps} = \{x\in A\,:\, \dist(x,\IR^d\setminus A) > \eps\}
}
(note that in general $(A^{-\eps})^\eps \neq A$).

We will use the following lemmas in the proof of Theorem \ref{t1}.
\begin{lemma}[Lemma 2.3 of \cite{Be03}] The
function $h_{A,\eps}$ as defined above
has the following properties:
\ban{   
  (i)&\enskip\text{$h_{A,\eps}(w)=1$ for all $w\in A$,}\label{29}\\
  (ii)&\enskip\text{$h_{A,\eps}(w)=0$ for all
                    $w\in\IR^d\setminus A^\eps$,}\label{30}\\
  (iii)&\enskip \text{$0\leq h_{A,\eps}(w) \leq 1$ for all
                     $w\in A^\eps\setminus A$,}\label{31}\\
 (iv)&\enskip \text{$|\nabla h_{A,\eps} (w)|\leq 2\eps^{-1}$ for all
                   $w\in \IR^d$,}\label{32}\\
 (v) &\enskip \text{$|\nabla h_{A,\eps}
    (v)-\nabla h_{A,\eps} (w)|\leq 8|v-w|\eps^{-2}$ for all
$v,w\in \mr$}.\label{33}
}
\end{lemma}

\begin{lemma}[\cite{Ball1993}, \cite{Be03}]
We have
\ben{\label{35}
  \sup_{A\in \mathcal{A}}\max \{\P (Z\in A^\eps\setminus A),\P (Z\in
    A\setminus A^{-\eps})\}
    \leq 4d^{1/4} \eps,
}
and the dependence on $d$ in \eq{35} is optimal.
\end{lemma}

\begin{lemma}[Lemma 4.2 of \cite{FaRo14}]\label{lem4}
For any $d$-dimensional random vector $W$,
\ben{\label{34}
  d_c(\law (W), \law (Z)) \leq 4d^{1/4}\eps +
    \sup_{A\in\mathcal{A}}\babs{\E [h_{A,\eps}(W)]-\E [h_{A,\eps}(Z)]}.
}
\end{lemma}

\begin{lemma}[Lemma 4.3 of \cite{FaRo14}]\label{lem5}
For each map $a: \{1,\dots,d\}^k\rightarrow \IR$, we have
\ben{\label{104}
\int_{\mr}\left( \sum_{i_1,\dots, i_k=1}^d a(i_1,\dots, i_k)\frac{\phi_{i_1\dots i_k}(z)}{\phi(z)}  \right)^2\phi(z)dz 
\leq k! \sum_{i_1,\dots, i_k=1}^d \left( a(i_1,\dots, i_k)  \right)^2,
}
where $\phi(z)$ is the density of $d$-dimensional standard normal distribution and
\be{
\phi_{i_1\dots i_k}(z)=\partial^k \phi(z)/(\partial z_{i_1} \dots \partial z_{i_k}).
}
\end{lemma}

Now fix $\eps$ and a convex set $A\subset \IR^d$. It can be verified directly
that, defining
\ben{\label{303}
g_{A,\eps}(w,\tau)=-\frac{1}{2(1-\tau)} \int_{\IR^d}
    \big\{h_{A,\eps}(\sqrt{1-\tau}w+\sqrt{\tau}z)-\E [h_{A,\eps} (Z)] \big\}\phi(z)dz,
}
the function
\ben{\label{103}
   f_{A,\eps}(w)=\int_0^1 g_{A, \eps}(w,\tau) d\tau
}
is a solution to \eq{27} is (cf. \cite{Goetze1991}).
In what follows, we keep the dependence on $A$ and $\eps$
implicit and write $g=g_{A, \eps}, f=f_{A,\eps}$ and $h=h_{A,\eps}$. For real-valued functions on $\IR^d$ we
write $f_r(x)$ for $\partial f(x)/\partial x_r$, $f_{rs}(x)$ for
$\partial^2 f(x)/(\partial x_r\partial x_s)$ and so forth. We also write $g_r(w,\tau)=\partial g(w,\tau)/\partial w_r$ and so on. Moreover, let $\nabla g(w,\tau)=(g_1(w,\tau),\dots, g_d(w,\tau))^t$ 
and let $\Delta g(w,\tau)=\sum_{r=1}^d g_{rr}(w,\tau)$.

Using this notation and the integration by parts formula, we have for $1\leq r,s,t\leq d$ that
\besn{   \label{36}
  g_{rs} (w,\tau) & =
    - \frac{1}{2\tau}
        \int_{\IR^d} h (\sqrt{1-\tau}w+\sqrt{\tau}z) \phi_{rs} (z) dz \\
   &=  \frac{1}{2\sqrt{\tau}} \int_{\IR^d}
     h_s(\sqrt{1-\tau}w+\sqrt{\tau}z) \phi_r (z) dz
}
and
\besn{\label{102}
  g_{rst} (w,\tau) & =
     \frac{\sqrt{1-\tau}}{2\tau^{3/2}}
        \int_{\IR^d} h (\sqrt{1-\tau}w+\sqrt{\tau}z) \phi_{rst} (z) dz \\
   &=  \frac{\sqrt{1-\tau}}{2\sqrt{\tau}} \int_{\IR^d}
     h_{jk}(\sqrt{1-\tau}w+\sqrt{\tau}z) \phi_r (z) dz.
}

\begin{proof}[Proof of Theorem~\ref{t1}]
Fix $A\in\mathcal{A}$ and $\eps>0$ (to be chosen
later) and let $f=f_{A,\eps}$ be the solution to the Stein equation \eq{27}
corresponding to $h=h_{A,\eps}$ as defined by \eq{28}. Let
\ben{\label{kappa}
  \kappa := d_c(\law(W),\law(Z)).
}
To avoid confusion, we will always use $r,s,t$ to index the components of $d$-dimensional vectors.
Define
\be{
W_i:=W-X_{N_i},\quad W_{ij}:=W-X_{N_{ij}},\quad W_{ijk}:=W-X_{N_{ijk}}.
}
By assumption \eq{t1-1} and because $\E (X_i)=0$, we have
\bes{
&- \E [ W^t \nabla g(W,\tau) ]
=-\sum_{r=1}^d \E [ (W)_r g_r (W,\tau) ]\\
&=-\sum_{r=1}^d \sum_{i=1}^n \E [(X_i)_r g_r(W, \tau) ]
=\sum_{r=1}^d \sum_{i=1}^n \E \{ (X_i)_r [g_r(W_i, \tau) - g_r (W,\tau)] \}.
}
Here and subsequently, we use $(X)_r$ to denote the $r$th component of a vector $X$.
By assumption \eq{t1-1} and $\Cov(W)=\II_d$, we have $\sum_{i=1}^n \sum_{j\in N_i} \E \{ (X_i)_r (X_j)_s \} =\delta_{rs}$ where $\delta$ denotes the Kronecker delta.
This implies that
\be{
\E [\Delta g(W,\tau)] =\E \big\{ \sum_{r,s=1}^d \sum_{i=1}^n \sum_{j\in N_i} (\E [(X_i)_r(X_j)_s]) g_{rs} (W,\tau)\big\}.
}
Adding and subtracting the corresponding terms, and using \eq{t1-2}, we have
\bes{
&\E \big\{\Delta g(W,\tau) -W^t \nabla g(W,\tau)   \big\}\\
&= \E \big\{ \sum_{r=1}^d \sum_{i=1}^n (X_i)_r
 \big[g_r(W_i,\tau) -g_r(W,\tau) +\sum_{s=1}^d\sum_{j\in N_i} (X_j)_s g_{rs} (W,\tau)  \big] \big\}\\
&\quad + \E \big\{ \sum_{r,s=1}^d \sum_{i=1}^n \sum_{j\in N_i} (\E [(X_i)_r(X_j)_s] -(X_i)_r (X_j)_s) [g_{rs} (W,\tau)-g_{rs}(W_{ij},\tau) ] \big\}\\
&=: R_1(\tau)+R_2(\tau).
}
Taking $g(w,\tau)=g_{A,\eps}(w,\tau)$ in \eq{303}, it follows from \eq{27} and \eq{103} that
\ben{\label{202}
\E [ h(W) ]-\E [h(Z)]=\int_0^1 (R_1(\tau)+R_2(\tau))d\tau.
}
In the following we will first give an upper bound for $|\int_0^1 R_1(\tau) d\tau|$,
and then argue that an upper bound for $|\int_0^1 R_2(\tau) d\tau|$ can be derived similarly.

To estimate $\int_0^1 R_1(\tau) d\tau$, we consider the cases $\eps^2< \tau\leq 1$ and $0< \tau\leq \eps^2$ separately.
For the first case, we use the first expression of $g_{rs}(w,\tau)$ in \eq{36}, Taylor's expansion
\be{
f(x+a)-f(x)=\int_0^1 a^t \nabla f(x+ua) du=\E \{a^t \nabla f(x+U a) \},
}
and the integration by parts formula, and get
\bes{
\int_{\eps^2}^1 R_1(\tau)d\tau
&= \E \big\{ \sum_{r,s=1}^d \sum_{i=1}^n \sum_{j\in N_i} \int_{\eps^2}^1 \frac{1}{2\tau} \int_{\mr} 
[h(\sqrt{1-\tau}(W-U X_{N_i})+\sqrt{\tau}z)-h(\sqrt{1-\tau}W+\sqrt{\tau}z)] \\
&\kern6em \times(X_i)_r (X_j)_s \phi_{rs} (z) dz d\tau \big\}\\
&= \E \big\{ \sum_{r,s,t=1}^d \sum_{i=1}^n \sum_{j,k\in N_i} \int_{\eps^2}^1 \frac{\sqrt{1-\tau}}{2\tau^{3/2}} \int_{\mr}
h(\sqrt{1-\tau}W +\sqrt{\tau}z-\sqrt{1-\tau}UVX_{N_i}) \\
&\kern6em \times U( X_i)_r(X_j)_s(X_k)_t \phi_{rst}(z)dzd\tau \big\},
}
where $U$ and $V$ are independent random
variables distributed uniformly on $[0,1]$.
By writing $h(\sqrt{1-\tau}W +\sqrt{\tau}z-\sqrt{1-\tau}UVX_{N_i})$ as a sum of differences, 
and using the independence assumption \eq{t1-3}, we have
\be{
\int_{\eps^2}^1 R_1(\tau)d\tau =R_{1,1}+R_{1,2}+R_{1,3}+R_{1,4}
}
where
\bes{
R_{1,1}&= \E \big\{ \sum_{r,s,t=1}^d \sum_{i=1}^n \sum_{j,k\in N_i} \int_{\eps^2}^1 \frac{\sqrt{1-\tau}}{2\tau^{3/2}} \int_{\mr}
\big[ h(\sqrt{1-\tau}W +\sqrt{\tau}z-\sqrt{1-\tau}UVX_{N_i}) \\
&\kern6em - h(\sqrt{1-\tau}W_{ijk}+\sqrt{\tau}z) \big]  U( X_i)_r(X_j)_s(X_k)_t \phi_{rst}(z)dzd\tau \big\},
}
\bes{
R_{1,2}&= \E \big\{ \sum_{r,s,t=1}^d \sum_{i=1}^n \sum_{j,k\in N_i} \int_{\eps^2}^1 \frac{\sqrt{1-\tau}}{2\tau^{3/2}} \int_{\mr}
\big\{\E [h(\sqrt{1-\tau}W_{ijk}+\sqrt{\tau}z) ]\\
&\kern6em -\E [ h(\sqrt{1-\tau}W+\sqrt{\tau}z)] \big\}  U( X_i)_r(X_j)_s(X_k)_t \phi_{rst}(z)dzd\tau \big\},
}
\bes{
R_{1,3}&= \E \big\{ \sum_{r,s,t=1}^d \sum_{i=1}^n \sum_{j,k\in N_i} \int_{\eps^2}^1 \frac{\sqrt{1-\tau}}{2\tau^{3/2}} \int_{\mr}
\big\{\E [h((\sqrt{1-\tau}W+\sqrt{\tau}z)] \\
&\kern6em -\E [h(\sqrt{1-\tau}Z+\sqrt{\tau}z)] \big\}  U( X_i)_r(X_j)_s(X_k)_t \phi_{rst}(z)dzd\tau \big\},
}
\bes{
R_{1,4}&= \E \big\{ \sum_{r,s,t=1}^d \sum_{i=1}^n \sum_{j,k\in N_i} \int_{\eps^2}^1 \frac{\sqrt{1-\tau}}{2\tau^{3/2}} \int_{\mr}
 h(\sqrt{1-\tau}Z +\sqrt{\tau}z) \\
&\kern6em \times  U( X_i)_r(X_j)_s(X_k)_t \phi_{rst}(z)dzd\tau \big\},
}
where $Z$ is an independent $d$-dimensional standard Gaussian random vector.

By the properties of $h$ in \eq{29} and \eq{30}, the boundedness condition \eq{t1-4} and the independence assumption \eq{t1-3},
\bes{
|R_{1,1}|&\leq \E \big\{ \sum_{r,s,t=1}^d \sum_{i=1}^n \sum_{j,k\in N_i} \int_{\eps^2}^1 \frac{\sqrt{1-\tau}}{4\tau^{3/2}} \int_{\mr}
\I(dist(\sqrt{1-\tau} W_{ijk} +\sqrt{\tau}z, A^\eps \backslash A)\leq \sqrt{1-\tau} n_3 \beta) \\
&\kern6em \times \big|  ( X_i)_r(X_j)_s(X_k)_t \phi_{rst}(z) \big|dzd\tau \big\}.
}
By the boundedness condition \eq{t1-4}, the definition of $\kappa$ in \eq{kappa} and \eq{35},
\besn{\label{201}
&\E \big[ \I(dist(\sqrt{1-\tau} W_{ijk} +\sqrt{\tau}z, A^\eps \backslash A)\leq \sqrt{1-\tau} n_3 \beta) \big]\\
&\leq \E \big[ \I(dist(\sqrt{1-\tau}W  +\sqrt{\tau}z, A^\eps \backslash A)\leq 2\sqrt{1-\tau} n_3 \beta) \big]\\
&\leq 4d^{1/4} (\frac{\eps}{\sqrt{1-\tau}}+4n_3\beta) +2\kappa.
}
By the Cauchy-Schwartz inequality, \eq{104} and the boundedness condition \eq{t1-4},
\be{
\int_{\mr} \sum_{r,s,t=1}^d | ( X_i)_r(X_j)_s(X_k)_t \phi_{rst}(z)| dz \leq \sqrt{6} \beta^3.
}
Therefore, using $\int_{\eps^2}^1 \frac{1}{\tau^{3/2}}d\tau \leq C \frac{1}{\eps}$,
\ben{\label{r11}
|R_{1,1}|\leq C n \beta^3 n_1^2 \frac{1}{\eps} [d^{1/4} (\eps+n_3 \beta)+\kappa].
}
Here and in the remainder of the proof, 
$C$ denotes an absolute constant, which may differ from line to line.

By the same argument, $|R_{1,2}|$ has the same upper bound as $|R_{1,1}|$.

By the properties of $h$ in \eq{29} and \eq{30}, the definition of $\kappa$ in \eq{kappa}, and \eq{35},
\bes{
&\E [h(\sqrt{1-\tau}W+\sqrt{\tau}z)-h(\sqrt{1-\tau}Z+\sqrt{\tau}z) ]\\
&\leq \E [\I(\sqrt{1-\tau}W+\sqrt{\tau}z \in A^\eps) -\I(\sqrt{1-\tau}Z +\sqrt{\tau}z \in A)]\\
&= \E [ \I(\sqrt{1-\tau}W+\sqrt{\tau}z \in A^\eps) -\I(\sqrt{1-\tau}Z +\sqrt{\tau}z \in A^\eps)
+ \I(\sqrt{1-\tau}Z +\sqrt{\tau}z \in A^\eps \backslash A) ]\\
&\leq \kappa + 4 d^{1/4} \frac{\eps}{\sqrt{1-\tau}}.
}
By the same lower bound and a similar argument as for $R_{1,1}$, we can bound $R_{1,3}$ by
\ben{\label{r13}
|R_{1,3}|\leq C n \beta^3 n_1^2 \frac{1}{\eps} (d^{1/4} \eps +\kappa).
}
Using the first expression of $g_{rst}(w,s)$ in \eq{102}, we have
\be{
R_{1,4}=\E \big\{ \sum_{r,s,t=1}^d \sum_{i=1}^n \sum_{j,k\in N_i} \int_{\eps^2}^1  U
 \E[(X_i)_r (X_j)_s (X_k)_t] g_{rst}(Z,\tau) d \tau \big\}.
}
Observe that, from \eq{303},
\bes{
\E [g(Z+w,\tau)]&=-\frac{1}{2(1-\tau)}\int_{\mr} \big\{ \E [h(\sqrt{1-\tau}(Z+w)+\sqrt{\tau}z)] -\E [h(Z)] \big\}\phi(z)dz\\
&=-\frac{1}{2(1-\tau)}\int_{\mr} h(\sqrt{1-\tau}w+z)\phi(z)dz+\frac{1}{2(1-\tau)}\E [h(Z)]\\
&=-\frac{1}{2(1-\tau)}\int_{\mr} h(x)\phi(x-\sqrt{1-\tau}w)dx+\frac{1}{2(1-\tau)}\E [h(Z)].
}
Differentiating with respect to $w_r, w_s$ and~$w_t$, and evaluating at $w=0$, we obtain
\be{
\E [g_{rst}(Z,\tau)]=\frac{\sqrt{1-\tau}}{2}\int_{\mr} h(x)\phi_{rst}(x)dx.
}
Now with
\eq{104} and \eq{t1-4},
\ben{\label{r14}
|R_{1,4}|\leq Cn \beta^3 n_1^2 .
}

For the case $0< \tau\leq \eps^2$, we use the second expression of $g_{rs}(w,\tau)$ in \eq{36} and the Taylor expansion
\bes{
\int_0^{\eps^2} R_1(\tau) d\tau
&= -\E \big\{ \sum_{r,s=1}^d \sum_{i=1}^n \sum_{j\in N_i} \int_0^{\eps^2} \frac{1}{2\sqrt{\tau}}\int_{\mr}
 [h_s(\sqrt{1-\tau}(W-UX_{N_i})+\sqrt{\tau}z)\\
&\kern6em -h_s(\sqrt{1-\tau}W+\sqrt{\tau}z)] (X_i)_r(X_j)_s \phi_r (z)dzd \tau   \big\}   \\
&=\E \big\{ \sum_{r,s,t=1}^d \sum_{i=1}^n \sum_{j,k\in N_i} \int_0^{\eps^2} \frac{\sqrt{1-\tau}}{2\sqrt{\tau}} \int_{\mr}
 h_{st}(\sqrt{1-\tau}W+\sqrt{\tau}z-\sqrt{1-\tau}UVX_{N_i})\\
&\kern6em \times U (X_i)_r(X_j)_s(X_k)_t  \phi_r (z) dzd\tau \big\},
}
where we recall that $U$ and $V$ are independent random
variables distributed uniformly on $[0,1]$.
By \eq{29}, \eq{30}, \eq{33}, and \eq{t1-4},
\bes{
\Big| \int_0^{\eps^2} R_1(\tau) d \tau  \Big|
&\leq \frac{8}{\eps^2} \beta^2 n_1^2 \E \big\{ \sum_{r=1}^d \sum_{i=1}^n \int_0^{\eps^2} \frac{\sqrt{1-\tau}}{2\sqrt{\tau}} \int_{\mr}
\I(dist(\sqrt{1-\tau}W_i+\sqrt{\tau}z, A^\eps\backslash A)\leq \sqrt{1-\tau} n_1 \beta) \\
&\kern6em \times U |(X_i)_r \phi_r(z)| dzd\tau \big\}.
}
Much as in \eq{201},
\bes{
&\E [\I(dist(\sqrt{1-\tau} W_{i} +\sqrt{\tau}z, A^\eps \backslash A)\leq \sqrt{1-\tau} n_1 \beta)]\\
&\leq \E [\I(dist(\sqrt{1-\tau} W  +\sqrt{\tau}z, A^\eps \backslash A)\leq 2\sqrt{1-\tau} n_1 \beta)]\\
&\leq 4d^{1/4} (\frac{\eps}{\sqrt{1-\tau}}+4n_1\beta) +2\kappa.
}
Together with \eq{t1-1}, \eq{t1-4} and \eq{104}, we obtain
\be{
\Big| \int_0^{\eps^2} R_1(\tau) d \tau  \Big| \leq Cn \beta^3 n_1^2 \frac{1}{\eps} [d^{1/4}(\eps + n_1\beta)+\kappa],
}
and hence, from \eq{r11}, \eq{r13} and \eq{r14},
\ben{\label{203}
\Big|\int_0^1 R_1(\tau)d \tau \Big|\leq C n \beta^3 n_1^2 \frac{1}{\eps} [d^{1/4} (\eps+n_3 \beta)+\kappa].
}

Now we turn to bounding $|\int_0^1 R_2(s) ds|$. Much as for $R_1(s)$,
\bes{
\int_0^1 R_2(\tau) d\tau
&= - \E \big\{ \sum_{r,s,t=1}^d \sum_{i=1}^n \sum_{j\in N_i}\sum_{k\in N_{ij}}
\int_{\eps^2}^1 \frac{\sqrt{1-\tau}}{2\tau^{3/2}} \int_{\mr} h(\sqrt{1-\tau} W+\sqrt{\tau}z-\sqrt{1-\tau} U X_{N_{ij}})\\
&\kern6em \times (X_k)_t \big\{(X_i)_r(X_j)_s-\E[(X_i)_r(X_j)_s]\big\} \phi_{rst} (z)dzd\tau \big\}\\
&\quad - \E \big\{ \sum_{r,s,t=1}^d \sum_{i=1}^n \sum_{j\in N_i}\sum_{k\in N_{ij}}
\int_0^{\eps^2} \frac{\sqrt{1-\tau}}{2\sqrt{\tau}} \int_{\mr} h_{st}(\sqrt{1-\tau} W+\sqrt{\tau}z-\sqrt{1-\tau} U X_{N_{ij}})\\
&\kern6em \times (X_k)_t \big\{(X_i)_r(X_j)_s-\E[(X_i)_r(X_j)_s]\big\} \phi_{r} (z)dzd\tau \big\}
}
where $U$ is a independent random variable distributed uniformly on $[0,1]$.
By the same arguments used in bounding $|\int_0^1 R_1(\tau)d\tau|$,
we have
\ben{\label{204}
\Big|\int_0^1 R_2(\tau)d \tau \Big|\leq C n \beta^3 n_1 n_2 \frac{1}{\eps} [d^{1/4} (\eps+n_3 \beta)+\kappa].
}
By \eq{34}, \eq{202}, \eq{203} and \eq{204},
\ben{\label{205}
\kappa\leq 4d^{1/4} \eps + C n \beta^3 n_1 n_2 \frac{1}{\eps} [d^{1/4} (\eps+n_3 \beta)+\kappa].
}
The final bound \eq{t1-5} is obtained by
choosing $\eps=2C n \beta^3 n_1 n_2$ for the same $C$ as in \eq{205}, solving the recursive inequality \eq{205}
and observing that $d\leq n \beta^2 n_1$ from $\Cov(W)=\II_d$ and \eq{t1-4}.

\end{proof}

\section*{Acknowledgments}

This work is based on part of the author's Ph.D. thesis. He is grateful to his advisor, Louis H. Y. Chen, for his guidance. He would like to thank Adrian R\"ollin for helpful discussions.
He would also like to thank the referees for their helpful comments and suggestions which have significantly improved the presentation of this paper.
This paper was written under the financial support of NUS-Overseas Postdoctoral Fellowship from the National University of Singapore.

\setlength{\bibsep}{0.5ex}
\def\bibfont{\small}


\end{document}